\documentclass[a4paper]{amsart}

\usepackage[colorlinks]{hyperref}
\usepackage[utf8]{inputenc}
\usepackage{newunicodechar}
\usepackage{lmodern}
\usepackage{textcomp}
\usepackage{xfrac}
\usepackage{tikz}
\usetikzlibrary{matrix,arrows,decorations.pathmorphing}
\usepackage{enumitem}
\usepackage{ dsfont }
\usepackage{ amssymb }
\usepackage{amsthm}
\usepackage{listings}

\newtheorem{satz}{Satz}
\newtheorem{thm}[satz]{Theorem}
\newtheorem{lem}[satz]{Lemma}
\newtheorem{cor}[satz]{Corollary}
\newtheorem{pro}[satz]{Proposition}

\newtheorem{defi}[satz]{Definition}

\theoremstyle{remark}
\newtheorem{rem}{Remark}

\renewcommand{\O}{\mathcal{O}}
\renewcommand{\P}{\mathds{P}}

\renewcommand{\H}{\textup{H}}

\newcommand*{\defeq}{\mathrel{\vcenter{\baselineskip0.5ex \lineskiplimit0pt
			\hbox{\scriptsize.}\hbox{\scriptsize.}}}%
	=}
\newcommand*{\eqdef}{=\mathrel{\vcenter{\baselineskip0.5ex \lineskiplimit0pt
			\hbox{\scriptsize.}\hbox{\scriptsize.}}}%
	}

\newcommand\restr[2]{
	{\left.\kern-\nulldelimiterspace#1\vphantom{\big|}\right|_{#2}}
}

\let\enumerateO\enumerate
\let\endenumerateO\endenumerate
\renewenvironment{enumerate}{\enumerateO\setlength{\parskip}{0em}\setlength{\parindent}{1em}}{\endenumerateO}

\makeatletter
\newcommand{\mylabel}[2]{(#2)\def\@currentlabel{#2}\label{#1}}
\makeatother

\newcommand{\Z}{\ensuremath{\mathds{Z}}}
\newcommand{\Q}{\ensuremath{\mathds{Q}}}
\newcommand{\R}{\ensuremath{\mathds{R}}}
\newcommand{\C}{\ensuremath{\mathds{C}}}
\newcommand{\Nn}{\ensuremath{\mathds{N}_{0}}}

\newcommand{\GL}{\ensuremath{GL}}

\newcommand{\inj}{\hookrightarrow}

\newcommand{\fN}{\ensuremath{\mathcal{N}}}

\newcommand{\fS}{\ensuremath{\mathcal{S}}}
\newcommand{\fT}{\ensuremath{\mathcal{T}}}

\DeclareMathOperator{\tr}{tr}

\DeclareMathOperator{\h}{h}

\DeclareMathOperator{\Aut}{Aut}

\title[]{On the definition of irreducible holomorphic symplectic manifolds and their singular analogs}
\date{\today}

\author{Martin Schwald}
\address{Martin Schwald, Universität Duisburg-Essen, Fakultät für Mathematik, Thea-Leymann-Str.\,9, 45127 Essen}
\email{\href{mailto:martin.schwald@uni-due.de}{martin.schwald@uni-due.de}} \urladdr{\url{http://www.esaga.uni-due.de/martin.schwald/}}
\date{\today}

\begin{document}

\begin{abstract}
In the definition of irreducible holomorphic symplectic manifolds the condition of being simply connected can be replaced by vanishing irregularity.  We discuss holomorphic symplectic, finite quotients of complex tori with $\h^0(X,\,\Omega^{[2]}_X)=1$ and their Lagrangian fibrations. Neither $X$ nor the base can be smooth unless $X$ is a $2$-torus.
\end{abstract}

\maketitle
\setcounter{tocdepth}{1}
\tableofcontents

\section{Introduction and main result}

\emph{Irreducible (holomorphic) symplectic manifolds} became popular objects to study as together with complex tori and Calabi-Yau manifolds they appear as factors in the Beauville-Bogomolov decomposition theorem \cite[Th\'{e}or\`{e}m~2]{Bea83}.
We found the following equivalent condition to Beauville's original definition \cite[p.~763f]{Bea83}.

\begin{thm}
\label{ihs}
Let $X$ be a compact Kähler manifold such that $\H^0(X,\,\Omega^2_X)\cong\C$ is generated by a holomorphic symplectic form.
Then $\h^1(X,\,\O_X)=0$ if and only if $X$ is simply connected, and therefore an irreducible symplectic manifold.
\end{thm}

The condition $\h^1(X,\,\O_X)=0$ naturally fits with definitions of singular analogs of irreducible symplectic manifolds discussed in the literature.
We recall them in Section~\ref{symplectic} to make this point precise.

To prove Theorem~\ref{ihs}, the main situation to consider consists of smooth finite quotients of complex tori.
Motivated by this and examples of Matsushita, we study symplectic torus quotients $X$ with $\h^0(X,\,\Omega^{[2]}_X)=1$ and describe their Lagrangian fibrations.
If $X$ is not a $2$-torus, it cannot be smooth by an application of results of Hiss--Szczepa\'nski and Lutowski and any Lagrangian fibration $X\to B$ is similarly built like Matsushita's example \cite[p.~7f]{Mat01}.
In particular $B$ is singular and $K_B$ torsion, see Theorem~\ref{lag2}.
We conclude with some singular examples constructed with the assistance of GAP \cite{GAP4}.

\subsection*{Acknowledgment}
The author thanks Daniel Greb, Christian Lehn and Andreas Demleitner for discussions and support.

\section{Singular symplectic varieties}
\label{symplectic}

Let $X$ be a normal complex variety, i.e. an irreducible and reduced Hausdorff complex space.
Holomorphic $p$-forms on the smooth locus $X_{reg}$ are called \emph{reflexive differential forms}.
Their associated sheaves on $X$ are denoted by $\Omega^{[p]}_X$ and we get
\[
  \Omega^{[p]}_X\cong i_*\Omega^p_{X_{reg}}\cong(\Omega^p_X)^{\ast\ast}\quad\text{and}\quad\H^0(X,\,\Omega^{[p]}_X)\cong\H^0(X_{reg},\,\Omega_{X_{reg}}^p)
\]
for $i\colon X_{reg}\inj X$ the inclusion.
Reflexive forms satisfy good pullback properties, in particular the extension theorem of Kebekus--Schnell shows that when $X$ has rational singularities, then they always extend to every resolution of $X$ \cite[Corollary~1.8]{KS18}.
In this terminology Beauville's definition of \emph{symplectic varieties} \cite[Definition~1.1]{Bea00} can be reformulated as follows.

\begin{defi}
A \emph{symplectic variety} is a Kähler variety with rational singularities and an $\omega\in\H^0(X,\,\Omega^{[2]}_X)$ that is non-degenerate at every smooth point of $X$, i.e. a (holomorphic) \emph{symplectic form} on the smooth locus.
\end{defi}

There are the following two widely studied classes of symplectic varieties $X$, which share many properties with irreducible symplectic manifolds. 

\begin{enumerate}
\item[(a)]\label{b}
$(X,\omega)$ compact symplectic variety with $\h^1(X,\,\O_X)=0$ and $\h^0(X,\,\Omega^{[2]}_X)=1$
\item[(b)]\label{a}
$(X,\omega)$ compact symplectic variety such that for every \emph{quasi-\'{e}tale} morphism $f\colon X'\to X$ the pullback $f^*\omega\in\H^0(X',\,\Omega^{[2]}_{X'})$ generates the exterior algebra of reflexive forms on $X'$, i.e.~$\bigoplus_{p\in\Nn}\H^0(X',\,\Omega^{[p]}_{X'})=\C\left[f^*\omega\right]$,
\end{enumerate}
where \emph{quasi-\'{e}tale} means that $f$ is a finite surjective morphism between normal complex varieties that is \'{e}tale outside of an analytic subset $Z\subset X'$ of codimension at least two.

Definition~(a) is older and was, often with stronger assumptions to the singularities of $X$ or assuming projectivity, first studied by Namikawa and Matsushita.
Definition~(b) is the class of symplectic varieties appearing in a singular version of the Beauville-Bogomolov decomposition theorem, which was proven in the projective case \cite[Theorem~1.5]{HP19} building upon \cite{GKP16, GGK17, DG18, Dru18}.
It generalizes the definition of irreducible symplectic manifolds: For smooth $X$ the map $f$ is \'{e}tale by the purity of the branch locus, then we only need to compare the holomorphic Euler characteristics applying \cite[Proposition~3]{Bea83}.

Building on results of Matsushita, the author of the present paper compared the geometric properties of both definitions \cite{Sch17}.
Symplectic varieties of both classes share many similarities with irreducible symplectic manifolds and those of definition~(a) remain interesting as there turned out to be a nice moduli theory for them \cite{BL18}.

The question if definition~(a) is also a generalization of irreducible symplectic manifolds apparently remained unclear in the literature, see for example \cite[Lemma~3.3 and preceding paragraph]{BL18} and \cite[page 18, last paragraph]{Per19}, but now we can finally answer it positively by Theorem~\ref{ihs}.

\begin{rem}
The reader is advised to be careful, as the terminology used in the literature to refer to the above definitions~(a) and (b) as well as to other classes of symplectic varieties is inconsistent.
\end{rem}

\section{An application of the decomposition theorem}

As a first structural result towards Theorem~\ref{ihs} we prove the following Proposition, which is related to \cite[Proposition~A.1]{HN11}.

\begin{pro}
\label{simplyconnected}
Let $(X,\omega)$ be a smooth symplectic variety with $\h^1(X,\,\O_X)=0$ and $\h^0(X,\,\Omega^{[2]}_X)=1$. Then $X$ is either simply connected or a smooth quotient of a complex torus by a finite group of biholomorphic automorphisms.
\end{pro}
\begin{proof}
By the Beauville-Bogomolov decomposition theorem, there is a finite \'{e}tale covering $\pi\colon\hat X\to X$ that splits as a product $\hat{X}\cong T\times X'\times Y$, where $T$ is a complex torus, $X'\defeq\prod_{i=1}^{k} X_i$ is the product of irreducible symplectic manifolds $X_i$, and $Y$ a product of at least $3$-dimensional Calabi-Yau manifolds \cite[Th\'{e}or\`{e}m~2]{Bea83}.
By going to a finite \'{e}tale covering of $\hat X$ we may assume $\pi$ to be a Galois covering, so $X\cong \hat X/G$ for $G\subset\Aut\hat X$ a group of biholomorphic automorphisms of $\hat X$ with $|G|=\deg\pi$, \cite[Proposition~3]{Bea83a}.

We can identify $\H^0(X,\,\Omega^2_X)$ with the space of $G$-invariant holomorphic $2$-forms $\H^2(\hat{X},\,\Omega^2_{\hat{X}})^G$, hence the latter one is generated by the pullback $\pi^*\omega$ of the symplectic form.
As $\h^0(X',\,\Omega^1_{X'})=\h^0(Y,\,\Omega^1_Y)=\h^0(Y,\,\Omega^2_Y)=0$, we get by the Künneth formula a decomposition $\pi^*\omega=\pi_T^*\eta+\pi_{X'}^*\omega'$, where $\pi_T\colon\hat X\to T$ and $\pi_{X'}\colon\hat{X}\to X'$ are the projections, and $\eta\in\H^0(T,\,\Omega^2_T)$, $\omega'\in\H^0(X',\,\Omega^2_{X'})$ are $2$-forms on the factors.
As $\pi$ is \'{e}tale, $\pi^*\omega$ is also non-degenerate, hence $Y$ is trivial and the forms $\eta$ and $\omega'$ are non-degenerate as well.

For every $f\in\Aut(\hat X)$ there are automorphisms $g\in\Aut(T)$ and $h\in\Aut(X')$ such that $f=(g,h)$
\cite[p.~8, Lemma]{Bea83a}, so $\pi_T^*\eta$ and $\pi_{X'}^*\omega'$ are $G$-invariant holomorphic $2$-forms on $\hat{X}$.
As $\h^0(X,\,\Omega^2_X)=1$, it follows that either $T$ or $X'$ has to be trivial. If $X'$ is trivial then $X\cong T/G$ and the proof is complete.

Hence we can assume from now on that $T$ is trivial, so $X\cong X'/G$. In this case, we show that every automorphism $f\in G\subset\Aut(X')$ has a fixed point. For this, let $\omega_i$ be the pullbacks of the symplectic forms of the factors $X_i$ to $X'$.
By \cite[p.~762f, Propositions~3, 4]{Bea83} and the Künneth formula, $\bigoplus_{p\in\Nn}\H^0(X',\,\Omega^{p}_{X'})$ is generated by the wedge products of the $\omega_i$. In particular $\H^{j,0}(X')=0$ for all odd $j$.
As $(X,\omega)$ is symplectic, $f^*$ preserves a symplectic form on $X'$. By rescaling the $\omega_i$ we may assume it to be the sum of the $\omega_i$.

By \cite[p.~10, b)+c)]{Bea83a} the automorphism $f$ acts on $X'$ by first possibly permuting isomorphic factors $X_i$ and then applying automorphisms $f_i\in\Aut(X_i)$ on each factor.
In particular there is a permutation $\sigma\in\fS_k$ such that for all $i$ we have that $f^*\omega_i$ is a multiple of $\omega_{\sigma(i)}$. As $f^*$ preserves the sum of the $\omega_i$, it simply permutes the $\omega_i$, as well as their wedge products.
Therefore $\tr(f^*|_{H^{j,0}(\hat{X})})$ is zero for all odd $j$, non-negative for all even $j$ and it equals one for $j=0,2\dim(X)$.
Thus $f$ has a fixed point by the holomorphic Lefschetz fixed point formula \cite[p.~426]{GH}.
As we assumed $X$ to be smooth, it follows that $G$ acts trivially on $X'$, so $X\cong X'$ is simply connected.
\end{proof}

\begin{rem}
Proposition~\ref{simplyconnected} has weaker assumptions than \cite[Proposition~A.1]{HN11}, which asserts that a complex projective manifold $X$ is simply connected if it is symplectic and the symplectic form generates the exterior algebra of global holomorphic differential forms $\bigoplus_{p\in\Nn}\H^0(X,\,\Omega^{p}_{X})$. It was noted that the published proof is incomplete, as the argument in \cite[Proposition~A.1]{HN11} to exclude the torus factor is flawed. This can easily be repaired for example by first arguing like above and then using that under the given stronger assumptions one can exclude the case where $X=T/G$ is a smooth torus quotient by comparing the holomorphic Euler characteristics: $\deg\pi\cdot\chi(X,\,\O_X)=\deg\pi\cdot(\frac{1}{2}\dim X+1)\neq0=\chi(T,\,\O_T)$.

However, in the situation of Proposition~\ref{simplyconnected} there is no obvious reason why $\chi(X,\,\O_X)$ should be non-zero when $\dim X\geq6$. This was the motivation of the present article. The situation where $X=T/G$ is a smooth quotient of a complex torus needs a deeper analysis to see that it does not occur, which we will do in the next two sections.
\end{rem}

\section{Complex torus quotients}
\label{ctq}
We use here the term \emph{(complex) torus quotient} for quotients $X\defeq T/G$ where $T$ is a complex torus and $G$ a finite group acting by biholomorphisms on $T$.
For a torus quotient the quotient map $T\to X$ is a finite morphism.
Every $g\in G$ acts on the universal covering $\C^n$ of $T$ as the composition of a linear endomorphism $L(g)$ and a translation $t(g)$.
Then $L\colon G\to\GL_{\C}(\C^n)$ is a group homomorphism, we call it the induced \emph{analytic representation} of the group action, cf. \cite[p.10]{BiLa}.
Then $\ker L$ is the set of translations in $G$.
Up to an isogeny of $T$, we may assume $G$ not to contain non-trivial translations and therefore $L$ to be faithful.
This way $L$ becomes a faithful linear representation of $G$.
As $G$ is finite, $L(g)$ has finite order for every $g\in G$.
In particular $L(g)$ is diagonizable and all its eigenvalues are roots of unity.
Moreover, $X$ has finite quotient singularities, so for every $k\in\Nn$ we have $\H^0(X,\,\Omega^{[k]}_X)\cong\H^0(T,\,\Omega^k_T)^G$, 
which can be identified with the space of holomorphic $k$-forms on $\C^n$ that are invariant under every translation and every $L(g)$ for $g\in G$.

\subsection*{Representation theory}
We denote the character of a complex representation $\rho$ of a finite group $G$ by $\chi_{\rho}$, the inner product of characters by $(\cdot|\cdot)$ \cite[Section~2.3]{Ser77} and the trivial representation of $G$ by $\mathbb{I}_G$.
When $\rho$ is irreducible, it can be of real, complex or quaternionic type, depending on its Frobenius-Schur-indicator $\iota(\rho)\defeq\frac{1}{|G|}\sum_{g\in G}\chi_{\rho}(g^2)$ being $1$, $0$ or $-1$, respectively  \cite[p.108--109]{Ser77}.


\begin{lem}[Reflexive 1-forms on torus quotients]
\label{h1}
Let $X=T/G$ be an $n$-dimensional torus quotient with induced analytic representation $L$. Then $\H^0(X,\,\Omega^{[1]}_X)$ is isomorphic to the space of $v\in\C^n$ fixed under $L(g)$ for all $g\in G$. Thus $\h^0(X,\,\Omega^{[1]}_X)=(\chi_L|\chi_{\mathbb{I}_G})=\frac{1}{|G|}\sum_{g\in G}\tr L(g)$ counts the multiplicity of the trivial representation $\mathbb{I}_G$ as a component in $L$.
\end{lem}
\begin{proof}
The space $\H^0(X,\,\Omega^{[1]}_X)$ lifts to a space of holomorphic $1$-forms on $\C^n$.
A translation-invariant $1$-form on $\C^n$ can be written as $\sum a_i dz_i=d(\sum a_iz_i)\eqdef dv$.
Such a form is $G$-invariant if and only if for all $g\in G$ we have $L(g)^*dv=d(L(g)(v))=dv$, which is equivalent to $v$ being fixed by $L(g)$ for all $g\in G$.
By Maschke's Theorem $L$ splits into a direct sum of irreducible representations and a common fixed vector $v$ corresponds to to a trivial direct summand of $L$.
In terms of characters, the multiplicity of $\mathbb{I}_G$ can be calculated via the claimed formula \cite[Theorem~4]{Ser77}.
\end{proof}

\begin{lem}[Symplectic torus quotients]
\label{h2}
Let $X=T/G$ be an $n$-dimensional torus quotient with induced analytic representation $L$.
Then $\H^0(X,\,\Omega^{[2]}_X)$ is spanned by a symplectic form on $X$ if and only if $L$ is either an irreducible representation of quaternionic type or the direct sum of two complex conjugate irreducible representations of real or complex type.
In every case it follows that $\h^1(X,\,\O_X)=0$ except when $X$ is a $2$-torus.
\end{lem}
\begin{proof}
An irreducible representation $L$ preserves a non-trivial alternating bilinear form if and only if $L$ is of quaternionic type;
in this case this form is non-degenerate and unique up to a scalar \cite[Propositions~38(b)]{Ser77}.

We suppose now that $L$ is a reducible representation $L=L_1\oplus L_2$, where the $L_i\colon G\to\GL_{\C}(V_i)$ are sub-representations of positive degrees and $\C^n=V_1\oplus V_2$ is a decomposition invariant under $L(g)$ for all $g\in G$.
Then the decomposition $\bigwedge^2\C^n\cong \left(\bigwedge^2 V_1\right)\oplus (V_1\otimes V_2)\oplus\left(\bigwedge^2 V_2\right)$ is invariant under $\wedge^2L(g)$ for all $g\in G$.
When there is a $G$-invariant $2$-form $\omega$ on $X$, its three components are also $G$-invariant.
Hence $\H^0(X,\,\Omega^2_X)=\C\omega$ implies that only one component is non-zero.
As $\omega$ is symplectic it is non-degenerate, such that $\omega$ has to lie in the mixed part $V_1\otimes V_2$.
By the same argument applied to sub-representations, $L_1,L_2$ are seen to be irreducible.
Clearly neither $L_1$ nor $L_2$ can be of quaternionic type, as then we had $\h^0(X,\,\Omega^2_X)\ge2$.
Using \cite[Proposition~3 and Theorem~4]{Ser77} we can calculate
\begin{align*}
1=\h^0(X,\,\Omega^{[2]}_X)&=(\chi_{\wedge^2L}|\chi_{\mathbb{I}_G})\\
&=\frac{1}{2|G|}\sum_{g\in G}(\chi_{L_1}(g)+\chi_{L_2}(g))^2-(\chi_{L_1}(g^2)+\chi_{L_2}(g^2))\\
&=\frac{1}{2}(\chi_{L_1}|\overline{\chi_{L_1}})+(\chi_{L_1}|\overline{\chi_{L_2}})+\frac{1}{2}(\chi_{L_2}|\overline{\chi_{L_2}})-\frac{1}{2}\iota(L_1)-\frac{1}{2}\iota(L_2)\big)\\
&=\begin{cases}
(\chi_{L_1}|\overline{\chi_{L_2}})&\text{if $L_1$ and $L_2$ are both real or both complex}\\
0&\text{else}
\end{cases}
\end{align*}
and conclude $L_1=\overline{L_2}$.

By Lemma~\ref{h1} we get $\h^0(X,\,\Omega^{[1]}_X)=(\chi_L|\chi_{\mathbb{I}_G})$. This is zero unless $L$ is reducible and $L_1,L_2$ are both trivial, which is precisely the case when $T$ and $X$ are both $2$-tori. In any case, as $X$ has rational singularities every reflexive form on $X$ extends to a resolution where we can apply Hodge symmetry. Hence $\h^1(X,\,\O_X)=\h^0(X,\,\Omega^{[1]}_X)$ (cf. the proof of \cite[Proposition~6.9]{GKP16}) and the claim follows.
\end{proof}

\section{Smooth complex torus quotients}
\label{sctq}
A complex torus quotient $X=T/G$ is smooth if and only if every $g\in G$ acts fixed point free on $T$.
We call $X$ a \emph{smooth (complex) torus quotient} in that case.
In the literature smooth torus quotients also appear under the names \emph{(generalized) hyperelliptic manifolds/varieties} as they generalize hyperelliptic surfaces, despite not being a generalization of the notion of hyperelliptic curves.

\subsection*{Decomplexification}

For a complex representation $\rho\colon G\to \GL_{\C}(\C^n)$ its \emph{decomplexification} $\rho_{\R}\colon G\to\GL_{\R}(\R^{2n})$ is obtained by forgetting the complex structure on $\C^n$.
We get $\C\otimes\rho_{\R}\cong\rho\oplus\bar{\rho}$, hence two irreducible complex representations are $\R$-equivalent if and only if they are $\C$-equivalent up to a possible complex conjugation.

\subsection*{Holonomy representation}

Let $X=T/G$ be a smooth complex torus quotient with $T=\C^n/\Lambda$ for a complete lattice $\Lambda\subset\C^n$ and induced analytic representation $L$.
Then $L(g)$ acts on $\Lambda$ for every $g\in G$.
The induced actions of $G$ on $\Q\otimes_{\Z}\Lambda$ or $\R\otimes_{\Z}\Lambda$ are called rational or real \emph{holonomy representation}, respectively.
The real holonomy representation equals the decomplexification of the analytic representation $L$ of $G$ on $\C^n$.

\subsection*{Homogeneous representations}
When $K$ is a field, we call a representation $\rho$ of $G$ over $K$ \emph{homogeneous} if all irreducible subrepresentations of $\rho$ are equivalent. When $\rho$ is a complex representation, then its decomplexification $\rho_{\R}$ is homogeneous if and only if all irreducible subrepresentations of $L$ are equivalent up to a possible complex conjugation.

\begin{lem}
\label{rational}
Let $\rho$ be a rational representation of a finite group and $K\supset\Q$ a field. When $K\otimes\rho$ is homogeneous then $\rho$ is homogeneous as well.
\end{lem}
\begin{proof}
Let $\eta_1,\eta_2$ be two irreducible subrepresentations of $\rho$. Then the $K\otimes\eta_i$ are both powers of the same irreducible subrepresentation of $K\otimes\rho$.
Hence the $K\otimes\eta_i$ are equivalent and thus have the same characters.
As $\chi_{K\otimes\eta_i}=\chi_{\eta_i}$, it follows that the $\eta_i$ are equivalent.
Hence $\rho$ is homogeneous.
\end{proof}

\begin{thm}
\label{homo}
Let $X=T/G$ be a smooth torus quotient with induced analytic representation $L$. If the decomplexification $L_{\R}$ is \emph{homogeneous} then $L$ is trivial and $X$ is a complex torus.
\end{thm}
\begin{proof}
Let $N\defeq\ker L$ the set of $g\in G$ acting via translations on $T$.
This is a normal subgroup of $G$ and we get $X\cong (T/N)/(G/N)$.
Hence we may assume without loss of generality that no $g\in G$ acts via a non-trivial translation on $T$.
Equivalently, $L$ is a faithful representation.
When $L_{\R}$ is homogeneous, then the rational holonomy representation is also homogeneous by Lemma~\ref{rational}.
Generalizing results from Hiss--Szczepa\'nski \cite[p.40, Corollary]{HS91}, Lutowski proved that if the rational holonomy representation is homogeneous, then $X$ is a complex torus \cite[Theorem~1]{Lut18}.
\end{proof}

\begin{cor}
\label{primitive}
Every smooth symplectic torus quotient with $\h^0(X,\,\Omega^2_X)=1$ is a complex $2$-torus.
\end{cor}
\begin{proof}
Let $X=T/G$ be a symplectic torus quotient with $\h^0(X,\,\Omega^{2}_X)=1$.
Then by Lemma~\ref{h2} its analytic representation $L\colon G\to\GL_{\C}(C^n)$ is either irreducible or the direct sum of two complex conjugate irreducible representations $L_i\colon G\to\GL_{\C}(V_i)$, $i=1,2$.
In the second case, the decomplexifications of the $L_i$ are equivalent as real representations, so $L_{\R}\cong (L_1)_{\R}^{\oplus 2}$.
In particular $L_{\R}$ is homogeneous in both cases.
By Theorem~\ref{homo} it follows that $X$ can only be smooth if it is a complex torus.
For $n\defeq\dim X$ this implies $1=\h^0(X,\,\Omega^{[2]}_X)=\frac{n(n-1)}{2}$ and thus $n=2$.
\end{proof}

The proof of the main theorem is now easy:

\begin{proof}[Proof of Theorem~\ref{ihs}]
Let $X$ be a compact Kähler manifold such that $\H^0(X,\,\Omega^2_X)\cong\C$ is generated by a holomorphic symplectic form $\omega$.
If $\h^1(X,\,\O_X)=0$ then $X$ is by Proposition~\ref{simplyconnected} simply connected or a smooth torus quotient.
In the latter case it would be a $2$-torus by Corollary~\ref{primitive}, for which $\h^1(X,\,\O_X)=2\neq0$.
For the other direction, if $X$ is simply connected, then $\H^1(X,\,\C)=\pi_1(X)^{ab}$ is trivial and hence $\h^1(X,\,\O_X)=0$ by Hodge theory.
\end{proof}

\begin{rem}
We invoked the results of Hiss--Szczepa\'nski and Lutowski, which rely on the classification of finite simple groups.
The author of the present paper wondered if this line of reasoning can be replaced by more direct arguments.
We give examples showing that the naivest attempts fail.
\end{rem}

Let $X=T/G$ be a smooth torus quotient and $L$ its induced analytic representation. For each $g\in G$ to act fixed point free on $T$ it is necessary that $L(g)$ has $1$ as an eigenvalue, cf. \cite[Lemma~13.1.1.]{BiLa}.
Furthermore, finite groups $G$ for which a smooth torus quotient $X=T/G$ with $\h^1(X,\,\O_X)=0$ exists are \emph{primitive} in the terminology of Hiller--Sah \cite{HS86}.
The authors of the latter paper gave an equivalent criterion for a finite group to be primitive, which is also easy to compute with a computer \cite[Theorem~3.8]{HS86}.
It follows from Lemma~\ref{h2} that $G$ needed to be primitive if there was a non-trivial smooth torus quotient $X=T/G$ that is symplectic with $\h^0(X,\,\Omega^{2}_X)=1$.

\begin{rem}
There are examples of irreducible representations $L$ --- of any type --- of primitive finite groups $G$ such that all representing matrices have $1$ as an eigenvalue, see Section~\ref{examples}.
They give rise to examples of singular symplectic torus quotients $X=T/G$ with $\h^1(X,\,\O_X)=0$ and $\h^0(X,\,\Omega^{[2]}_X)=1$, but by Corollary~\ref{primitive} no set of translations $t(g)$ for $g\in G$ can be chosen to make the induced action of $G$ on $T$ free.

However, smooth symplectic torus quotients with $\h^1(X,\,\O_X)=0$ exist indeed. The smallest example was very recently discussed by Hiss--Lutowski--Szczepa\'nski \cite{HLS20} and comes from the direct sum of three non-equivalent irreducible representations of quaternionic type, hence $\h^0(X,\,\Omega^{[2]}_X)=3$ in that case.
\end{rem}

\section{Lagrangian fibrations of symplectic torus quotients}

\begin{defi}
Let $(X,\omega)$ be a symplectic variety and $f\colon X\to B$ a surjective morphism with connected fibers onto a normal Kähler variety $B$.
If the general fiber $F$ of $f$ has dimension $\frac{1}{2}\dim X$ and $\omega|_{F_{reg}}$ vanishes\footnote{Note that $F$ has canonical singularities with $F_{reg}=F\cap X_{reg}$, hence $\omega|_{F_{reg}}$ is well defined as a reflexive form on $F$, \cite[Lemma~28]{Sch17}}, then $f$ is called a \emph{Lagrangian fibration} of $X$.
\end{defi}

Matsushita gave an example of a Lagrangian fibration of a symplectic torus quotient $X$ with $\h^0(X,\,\Omega^{[2]}_X)=1$ over another singular torus quotient $B$, \cite[p.~7f]{Mat01}. 
This is interesting, as when $X$ is a symplectic variety of type (b), e.g. an irreducible symplectic manifold, then the base of a Lagrangian fibration is always Fano, and when it is smooth, it is biholomorphic to $\P^{\frac{1}{2}\dim X}$, \cite[Theorem~3]{Sch17} building upon \cite[Theorem~2~(3)]{Mat01} \cite[Theorem~1.2]{Hwa08}.

It is still open if for a Lagrangian fibration of a symplectic variety $X$ with $\h^1(X,\,\O_X)=0$, $\h^0(X,\,\Omega^{[2]}_X)=1$ over a smooth base $B$ we always have $B\cong\P^{\frac{1}{2}\dim X}$.
An attempt to construct a counterexample would be to find a variation of Matsushita's example with a smooth base.
However, we show in Theorem~\ref{lag2} that this is impossible.
We start with a minor variation of \cite[Corollary~2.2.]{DHP08}.

\begin{pro}
\label{lag1}
Let $X=T/G$ be an $n$-dimensional symplectic torus quotient with induced analytic representation $L$ and a surjective morphism $f\colon X\to B$ onto an $m$-dimensional normal complex variety $B$ with $0<m<n$.
Then there is an $(n-m)$-dimensional linear subspace $V\subset\C^{n}$ that is invariant under $L(g)$ for all $g\in G$ and the general fiber of the projection $\C^{n}\to B$ is a translation of $V$.
\end{pro}
\begin{proof}
We write $T=\C^{n}/\Lambda$ and denote by $p\colon\C^{n}\to T$ and $\pi\colon T\to X$ the projections.
There is a Zariski-open set $U\subset B_{reg}$, such that for all $b\in U$ the fiber $F=\pi^{-1}(f^{-1}(b))$ of $f\circ\pi$ is smooth of dimension $n-m$ and $(f\circ\pi)|_{\pi^{-1}(f^{-1}(U))}$ is a submersion of complex manifolds \cite[Theorems~1.19 and 1.22]{CV7}.
For each connected component $F_0$ of $F$ the normal bundle $\fN_{F_0|T}$ is trivial and by its defining exact sequence also the tangent bundle $\fT_{F_0}$.
Hence $F_0$ is a complex torus \cite[Corollary~2]{Wan54}.
Consider the Stein factorization of $f\circ\pi$, i.e a holomorphic map $\hat{f}\colon T\to\hat{B}$ with connected fibers onto a normal complex variety $\hat{B}$, and a finite map $\hat{\pi}$ such that $f\circ\pi=\hat{\pi}\circ{\hat{f}}$. 
The general fiber of $\hat{f}$ is a connected component of a general fiber of $f$ and thus an $(n-m)$-dimensional complex torus.
The restriction $\hat{f}|_{\hat{f}^{-1}(\hat{\pi}^{-1}(U))}$ is a submersion of complex manifolds with connected fibers, so a deformation for each of its fibers $E$.
It follows that each $E$ is a translation of an $(n-m)$-dimensional subtorus of $T$ \cite[Theorem~10.3]{Ueno75}.
Two subtori that differ only by a translation are equal.
Hence, as $\hat{\pi}^{-1}(U)$ is connected, all $E$ are translations of the same subtorus $T_0$ of $T$.
Then $V\defeq p^{-1}(T_0)$ is an $(n-m)$-dimensional linear subspace of $\C^{n}$.
The action of $G$ on $T$ permutes the fibers of $\hat{f}$, and thus the action of $G$ on $\C^{n}$ permutes the affine subspaces $p^{-1}(E)$.
Hence it follows that for all $g\in G$ the map $L(g)$ preserves $V$, which is the common direction space of the $p^{-1}(E)$ .
\end{proof}

\begin{thm}
\label{lag2}
Let $X=T/G$ be an $n$-dimensional symplectic torus quotient with $\h^0(X,\,\Omega^{[2]}_X)=1$, analytic representation $L$ and a Lagrangian fibration $f\colon X\to B$ onto a normal complex variety $B$.

Then $L$ is reducible and $f$ is induced by a projection onto one of the two $m$-dimensional linear subspaces that are invariant under $L(g)$ for all $g\in G$.
Hence $B$ is a torus quotient with irreducible analytic representation of real or complex type. In particular $B$ is singular unless $f$ is a fibration of a $2$-torus over an elliptic curve.
\end{thm}
\begin{proof}
Let $f\colon X\to B$ be a Lagrangian fibration of $X=T/G$ with $T=\C^{n}/\Lambda$. 
Then $\dim B=\frac{n}{2}$, \cite[Theorems~1.19]{CV7}.
By Proposition~\ref{lag1} there is an $\frac{n}{2}$-dimensional linear subspace $V\subset\C^{n}$ that is invariant under $L(g)$ for all $g\in G$.
In particular $L$ cannot be irreducible.
Hence by Lemma~\ref{h2} we get $\C^{n}=V_1\oplus V_2$ and $L=L_1\oplus L_2$ for two complex conjugate irreducible representation $L_i\colon G\to\GL_{\C}(V_i)$ of real or complex type.
Therefore, up to switching the indices, we have $V=V_1$ and $B$ is the quotient of $V_2/(\Lambda\cap V_2)$ modulo the restricted action of $G$, which has linear part $L_2$.
Thus the analytic representation of $B$ is irreducible of real or complex type.
Hence by Theorem~\ref{homo} or directly by \cite[Theorem~2]{Lut18} the base $B$ is singular unless $B$ is a complex torus.
The latter case is equivalent to $L_1$ and $L_2$ being trivial and thus $f$ being a fibration of a $2$-torus over an elliptic curve.
\end{proof}

\begin{rem}
It follows from Theorem~\ref{lag2} that every non-trivial Lagrangian fibration $X\to B$ of a symplectic torus quotient with $\h^0(X,\,\Omega^{[2]}_X)=1$ is built like Matsushita's example \cite[p.~7f]{Mat01}.
In particular $K_B$ is always torsion and there is no variant where the base $B$ is smooth.
In GAP Matsushita's example can be reproduced via the two faithful irreducible complex conjugate representations {\tt Irr(SmallGroup(108,22))[i]} for $i=17,18$.
\end{rem}

\section{Singular examples}
\label{examples}

Symplectic torus quotients $X$ with $\h^0(X,\,\Omega^{[2]}_X)=1$ up to $\dim X\leq4$ were extensively studied by Fujiki \cite{Fu82}.
To construct more examples one needs to search for irreducible representations $\rho$ of finite groups $G$.
For $\rho$ of quaternionic type consider $L=\rho$, otherwise $L=\rho\oplus\bar{\rho}$.
If $L$ preserves a complete lattice, which is not always the case, it is by Lemma~\ref{h2} the analytic representation induced by a symplectic torus quotient $X$ with $\h^0(X,\,\Omega^{[2]}_X)=1$.

Using the following functions one can check with GAP, \cite{GAP4}, if a finite group $G$ is primitive and if all representing matrices of its $i$-th irreducible representation have $1$ as an eigenvalue.

\begin{lstlisting}[frame=single,language=GAP,basicstyle=\small]
LoadPackage("SmallGrp");
LoadPackage("ctbllib");
LoadPackage("HAP");

prim:=function(G)
local p;
for p in PrimeDivisors(Order(G)) do
	if IsCyclic(SylowSubgroup(G,p))=true
		and GroupHomology(G,1,p)<>[]
		then return false;
	fi;
od;
return true;
end;

eig:=function(G,i)
local ev,j;
for j in [1..Length(Irr(G)[i])] do
	ev:=Eigenvalues(Irr(G)[i],j);
	if ev[Length(ev)]=0 then return false;fi;
od;
return true;
end;
\end{lstlisting}

For checking the latter condition time can be saved by only considering groups with trivial center.
We list below for each of the three types the smallest primitive groups with a faithful irreducible representation such that all representing matrices have $1$ as an eigenvalue.
\begin{itemize}
\item Real type: The first non-trivial example is the 3-dimensional irreducible representation $\rho$ of $\fS_4$. $L=\rho\oplus\bar{\rho}$ induces an action on $(\C/\Lambda)^6$ for any lattice $\Lambda$ and preserves the symplectic form $dz_1\wedge dz_4+dz_2\wedge dz_5+dz_3\wedge dz_6$.
\item Complex type: {\tt Irr(SmallGroup(216,153))[i]} for $i=9,10$\\
$G$ is a group with $216$ elements with a faithful irreducible representation $\rho$ of degree $8$ generated by the three linear maps
\begin{align*}
(z_1,\ldots,z_8)\mapsto&(z_3,\ldots,z_8,z_1,z_2)\\
(z_1,\ldots,z_8)\mapsto&\zeta_3(z_4,z_3,z_8,z_7,z_5,z_6,z_1,z_2)\\
(z_1,\ldots,z_8)\mapsto&(\zeta_3^2z_1,\zeta_3z_2,\zeta_3^2z_3,\zeta_3z_4,\zeta_3z_5,\zeta_3^2z_6,z_7,z_8)
\end{align*}
for $\zeta_3$ a primitive third root of unity.
$L=\rho\oplus\bar{\rho}$ induces an action on $\C^{16}/(\Z+\zeta_3\Z)^{16}$ preserving the symplectic form $dz_1\wedge dz_9+\ldots+dz_8\wedge dz_{16}$.
\item Quaternionic type: {\tt Irr(SmallGroup(1280, 1116311))[9]}\\
$G$ is a group with $1280$ elements with a faithful irreducible representation $\rho$ of degree $20$ generated by the two linear maps
\begin{align*}
(z_1,\ldots, z_{20})\mapsto&(z_5,\ldots z_{20},z_1,\ldots z_4)\\
(z_1,\ldots, z_{20})\mapsto&(-iz_1,iz_2,iz_3,-iz_4,-z_6,z_5,z_8,-z_7,-iz_{10},-iz_9,iz_{12},\\&iz_{11},-iz_{16},-iz_{15},-iz_{14},-iz_{13},-iz_{20},-iz_{19},iz_{18},iz_{17})
\end{align*}
$L=\rho$ induces an action on $\C^{20}/(\Z+i\Z)^{20}$ preserving the symplectic form $dz_1\wedge dz_2+\ldots+dz_{19}\wedge dz_{20}$.
\end{itemize}

\bibstyle{alpha}
\bibliographystyle{alpha}
\bibliography{general}

\end{document}